\documentclass[11pt]{article}
\usepackage[margin=1in]{geometry}
\usepackage{amsmath,amssymb,amsthm,mathtools}
\usepackage{enumitem,microtype}
\usepackage{authblk}
\usepackage[colorlinks=true,linkcolor=blue,citecolor=blue,urlcolor=blue]{hyperref}

\newtheorem{theorem}{Theorem}[section]
\newtheorem{proposition}[theorem]{Proposition}
\newtheorem{lemma}[theorem]{Lemma}
\newtheorem{corollary}[theorem]{Corollary}
\theoremstyle{definition}
\newtheorem{definition}[theorem]{Definition}
\newtheorem{example}[theorem]{Example}
\theoremstyle{remark}
\newtheorem{remark}[theorem]{Remark}

\newcommand{\TP}{\mathrm{TP}}
\newcommand{\TN}{\mathrm{TN}}
\newcommand{\ETP}{\mathrm{ETP}}
\newcommand{\R}{\mathbb{R}}

\title{Eternal Total Positivity and Matrix-Dependent Hadamard Thresholds}
\author[1]{Shaun M. Fallat\thanks{Email:
\href{mailto:Shaun.Fallat@uregina.ca}{Shaun.Fallat@uregina.ca}.}}
\author[1,2]{Seyed Mahmoud Manjegani\thanks{Corresponding author. Emails:
\href{mailto:Seyed.Manjegani@uregina.ca}{Seyed.Manjegani@uregina.ca};
\href{mailto:manjgani@iut.ac.ir}{manjgani@iut.ac.ir}.}}
\author[2]{Mohammadrahim Saeidi Dahaki\thanks{Email:
\href{mailto:m.saeedi@alumni.iut.ac.ir}{m.saeedi@alumni.iut.ac.ir}.}}
\affil[1]{Department of Mathematics and Statistics, University of Regina,
3737 Wascana Parkway, Regina, Saskatchewan S4S 0A2, Canada}
\affil[2]{Department of Mathematical Sciences, Isfahan University of
Technology, Isfahan 84156-83111, Iran}
\date{}

\begin{document}
\maketitle

\begin{abstract}
We study entrywise powers of totally positive matrices.  Our first question is
which matrices remain totally positive under every positive Hadamard power.
We call them \emph{eternally totally positive}.  Our second question concerns
a fixed totally positive matrix: what is the smallest exponent after which all
of its entrywise powers are totally positive?  We call this number the
Hadamard threshold of the matrix.  Unlike the usual critical exponent, this
threshold depends on the individual matrix.  We give a complete
characterization in order three, obtain explicit bounds for the transition
exponent, and construct symmetric matrices that realize every threshold in
$[0,1)$.  We also prove the basic transformation rules and continuity of the
threshold on the normalized $3\times3$ totally positive cone.  For higher
orders, we introduce the logarithmic core and use it to expand each minor near
the zero exponent.  This reduces the $4\times4$ problem to the full determinant
and the $3\times3$ submatrices.  For symmetric $4\times4$ matrices, we show
that a positive-definite logarithmic core makes the full determinant positive
for every positive exponent.  A two-parameter Hankel family then realizes
every threshold in $(1,2)$.  Hence the exact range of thresholds in the
symmetric order-four case is $[0,1)\cup(1,2)$; neither $1$ nor $2$ is attained.
We also give a recursive symmetric ETP criterion in every order and an explicit
classification for symmetric Toeplitz matrices of order four. In order five,
a certified example exhibits an intermediate-power obstruction in a nonprincipal
minor. An explicit order-three condition and a finite critical-point test
give a complete order-five test.
\end{abstract}



\section{Introduction}

For a matrix $A=(a_{ij})$ with positive entries and a real number $t>0$, its
Hadamard, or entrywise power is defined as
\[
 A^{\circ t}:=(a_{ij}^{t}).
\]
The relation between Hadamard powers and matrix positivity has been studied in
several settings.  The classical theorem of FitzGerald and Horn determines
which powers preserve positive semidefiniteness in a fixed dimension~\cite{FitzGeraldHorn1977}.  I
n the total positivity setting, Fallat and Johnson proved that the positive Hadamard powers
of a fixed totally positive matrix are totally positive for all sufficiently
large exponents.  They also established the basic order-three result
\cite{FallatJohnson2007}.  Later, Fallat, Johnson, and Sokal determined the
universal Hadamard critical exponents for general, symmetric, and Hankel TP
matrices, and produced sharp counterexamples \cite{FallatJohnsonSokal2017,
FallatJohnsonSokal2021}.

These results address a universal question: which exponents preserve total
positivity for every matrix in a given class?  We ask a different question.
For a fixed totally positive matrix $A$, how can we determine the set
\[
 \mathcal P(A):=\{t>0:A^{\circ t}\text{ is totally positive}\}?
\]
There are two natural questions about this set.  Is
$\mathcal P(A)=(0,\infty)$?  If not, what is the smallest number $T$ such that
$(T,\infty)\subseteq\mathcal P(A)$?  The $3\times3$ case is the natural place
to begin.  Hadamard powers preserve all minors of orders one and two, and after
a standard normalization only one exponential-polynomial inequality remains.

A related input-dependent threshold problem has recently been studied by
Alseidi, Alsaidi, and Garloff \cite{AlseidiAlsaidiGarloff2026} for positive
coefficientwise powers of Hurwitz-stable polynomials of degree five.
They prove that each such polynomial has a unique stability threshold in
$(0,1)$, with stability holding precisely above that threshold.  Their result
concerns Hurwitz stability, whereas we require strict positivity of every
minor of a fixed matrix.  In particular, our setting also admits the
zero-threshold case, which leads to the classification of eternal total
positivity.

In this paper, we develop this matrix-dependent point of view.  We state
clearly which arguments come from the known critical-exponent theory and which
conclusions are new.  We then determine the threshold range in order three
and in the symmetric order-four case.
Section~\ref{sec:toeplitz} applies classical infinite-divisibility tools to
ETP, derives an explicit symmetric Toeplitz order-four classification, and
identifies the nonprincipal obstruction in order five.

\section{Terminology and Background}\label{sec:background}

\begin{definition}
An $m\times n$ real matrix is \emph{totally positive} ($\TP$) if every minor is
strictly positive.  It is \emph{totally nonnegative} ($\TN$) if every minor is
nonnegative.  More generally, it is $\TP_r$ (respectively, $\TN_r$) if every
minor of order at most $r$ is positive (respectively, nonnegative).
\end{definition}

\begin{definition}[Eternal total positivity]
Let $A$ be a matrix with positive entries.  We say that $A$ is
\emph{eternally totally positive}, and write $A\in\ETP$, if
$A^{\circ t}\in\TP$ for every $t>0$.
Likewise, $A$ is \emph{eternally totally positive of order $r$}, written
$A\in\ETP_r$, if $A^{\circ t}\in\TP_r$ for every $t>0$.
\end{definition}

\begin{definition}[Hadamard threshold]
For a matrix $A$ with positive entries, define
\[
 \tau(A):=\inf\bigl\{T\geq0:
 A^{\circ t}\in\TP\text{ for every }t>T\bigr\},
\]
with the convention $\inf\varnothing=\infty$.  More generally,
\[
 \tau_r(A):=\inf\bigl\{T\geq0:
 A^{\circ t}\in\TP_r\text{ for every }t>T\bigr\}.
\]
\end{definition}

\begin{definition}[Hankel matrix]
A matrix $A=(a_{ij})$ is called a \emph{Hankel matrix} if its entries are
constant along the anti-diagonals.  Equivalently, there is a sequence
$(h_k)$ such that $a_{ij}=h_{i+j}$ for all $i,j$.
\end{definition}

If $A$ is totally positive, its sufficiently large Hadamard powers are totally
positive.  Thus the set in the definition of $\tau(A)$ is nonempty, and $A\in\ETP$ if and only if $\tau(A)=0$.
Indeed, any failure at an exponent $t_0>0$ forces $\tau(A)\geq t_0$.

\begin{remark}
The term ``eternal'' is used here to emphasize preservation under every
positive Hadamard power. Related all-power preservation phenomena occur in
the theory of infinitely divisible matrices and moment sequences
\cite{Berg2007, Bhatia2006}. We retain the adjective ``totally'' to
distinguish this property from other notions of positivity.
\end{remark}

\section{Previously established results}

We summarize the results most directly relevant to the present paper.

\begin{theorem}[Fallat--Johnson \cite{FallatJohnson2007}]
\label{thm:FJcritical}
If $A$ is $\TP_3$, then $A^{\circ t}$ is $\TP_3$ for every $t\geq1$.
Moreover, the lower bound $1$ cannot be replaced by a smaller universal
constant.  For each fixed totally positive matrix $A$, the matrix
$A^{\circ t}$ is totally positive for all sufficiently large $t$.
\end{theorem}

Thus $\tau(A)<\infty$ for every totally positive matrix $A$.  This result,
however, does not give a formula for the threshold of a particular matrix.

\begin{theorem}[Fallat--Johnson--Sokal
\cite{FallatJohnsonSokal2017,FallatJohnsonSokal2021}]
\label{thm:FJScritical}
The following universal critical-exponent statements hold.
\begin{enumerate}[label=\textup{(\roman*)}]
 \item For $\TP_3$ and $\TN_3$ matrices, the critical exponent is $1$.
 \item For symmetric $4\times4$ $\TP$ or $\TN$ matrices, it is $2$.
 \item For general $n\times n$ matrices with $n\geq4$, it is infinite.
 \item For symmetric $n\times n$ matrices with $n\geq5$, it is infinite.
 \item For Hankel $\TP_r$ or $\TN_r$ matrices, it is exactly $r-2$.
\end{enumerate}
\end{theorem}

In particular, for each fixed $t>1$, there is a $4\times4$ totally positive
matrix $A$ for which $A^{\circ t}$ is not even $\TN_4$.  There is no conflict
with eventual total positivity: the threshold depends on $A$ and is not
uniformly bounded on the whole class.

More recent papers consider entrywise functions that preserve total
positivity, as well as totally positive kernels and P\'olya frequency functions
\cite{BeltonEtAl2022,BeltonEtAl2023}.  These results show that an entrywise
transformation is very restricted when it must work in every dimension.  Our
problem is different: we fix the matrix and vary only the exponent.

\section{The normalized \texorpdfstring{$3\times3$}{3 by 3} problem}

Positive diagonal scaling is particularly useful here.  If $D_1,D_2$ are
positive diagonal matrices, then
$(D_1AD_2)^{\circ t}=D_1^{\circ t}A^{\circ t}D_2^{\circ t}$.
Consequently, positive diagonal scaling preserves the sign of every minor of
every positive Hadamard power.

The following two standard reductions are recalled for completeness.  They are
used explicitly in the proof of the order-three Hadamard-power theorem; see
\cite[in the proof of Theorem~5.2]{FallatJohnsonSokal2017} and the earlier treatment
in \cite{FallatJohnson2007}.

\begin{proposition}[Standard normalization]
Every $3\times3$ matrix with positive entries is diagonally equivalent to a
matrix of the form
\begin{equation}\label{eq:normalized}
 A=\begin{pmatrix}
 1&1&1\\
 1&a&b\\
 1&c&d
 \end{pmatrix}.
\end{equation}
The original matrix is eternally totally positive if and only if its normalized
representative is eternally totally positive.
\end{proposition}

For the ordering in \eqref{eq:normalized}, total positivity yields
\begin{equation}\label{eq:TPconditions}
 b>a>1,\quad c>a>1,\quad d>b,\quad d>c,\quad ad>bc,
\end{equation}
and
\begin{equation}\label{eq:detpositive}
 ad-bc-d+b+c-a>0.
\end{equation}
Every positive Hadamard power preserves the $1\times1$ and $2\times2$ minors
of a $\TP_2$ matrix.  Hence, in order three, only the determinant remains.
The preservation of minors of order at most two and the determinant formula
below are recorded in \cite[Proposition~5.1 and proof of
Theorem~5.2]{FallatJohnsonSokal2017}.

\begin{proposition}[Standard determinant reduction]
\label{prop:determinant-reduction}
Let $A$ be the normalized matrix in \eqref{eq:normalized} and assume that it is
totally positive.  Then $A\in\ETP$ if and only if
\begin{equation*}\label{eq:F}
 F_A(t):=(ad)^t+b^t+c^t-(bc)^t-a^t-d^t>0
 \qquad(t>0).
\end{equation*}
\end{proposition}

\begin{proof}
All minors of order one or two of $A^{\circ t}$ are positive for $t>0$.
Direct expansion gives $\det(A^{\circ t})=F_A(t)$, so the assertion follows.
\end{proof}

Set $x=\log a,~ y=\log b,~ z=\log c$, and $w=\log d$
and introduce the vectors
 $p=(x+w,y,z)$ and $q=(y+z,x,w)$.
Their coordinate sums are equal, and
\[F_A(t)=\sum_{j=1}^{3}e^{tp_j}-\sum_{j=1}^{3}e^{tq_j}.\]
This formula may suggest using classical majorization.  However, the
total-positivity inequalities show that the two vectors cannot be compared in
this way.  Indeed,
$x<\min(y,z)$, $w>\max(y,z)$, and $x+w>y+z$.  Thus $x+w$ is the largest
coordinate of $p$, while $x$ is the smallest coordinate of $q$.  Consequently,
neither $q\prec p$ nor $p\prec q$ holds.  We therefore study the
exponential inequality directly.

\section{Complete analysis in order three}

Taylor expansion at the origin gives $F_A(0)=F_A'(0)=0$ and
\[
 F_A^{(k)}(0)=
 (x+w)^k+y^k+z^k-(y+z)^k-x^k-w^k.
\]

A direct calculation also gives $ F_A''(0)=2(xw-yz)$,
so $xw-yz$ determines the sign near zero.  The next theorem shows that the same
quantity determines the conclusion for every $t>0$.

\begin{theorem}[Order-three classification]\label{thm:classification}
Let $A$ be a normalized $3\times3$ totally positive matrix as in
\eqref{eq:normalized}.  Then the following are equivalent:
\begin{enumerate}[label=\textup{(\roman*)}]
 \item $A$ is eternally totally positive;
 \item $F_A(t)>0$ for every $t>0$;
 \item
 \begin{equation}\label{eq:classification}
   (\log a)(\log d)\geq(\log b)(\log c).
 \end{equation}
\end{enumerate}
\end{theorem}

\begin{proof}
The equivalence of \textup{(i)} and \textup{(ii)} follows from
Proposition~\ref{prop:determinant-reduction}.
If (ii) holds, the expansion
\[
 F_A(t)=(xw-yz)t^2+O(t^3)
\]
implies $xw-yz\geq0$, proving (iii).

For the converse, list the exponents of $F_A$ in increasing order.  The
smallest exponent is $x$, with negative coefficient; the next two are $y,z$,
with positive coefficients; the next two are $w,y+z$, with negative
coefficients; and the largest is $x+w$, with positive coefficient.  Equalities
inside either same-sign pair do not affect the argument.  Hence the coefficient
sequence has exactly three sign changes.  Laguerre's rule of signs for
exponential polynomials \cite[Part~V, Chapter~1]{PolyaSzego1976} implies that
$F_A$ has at most three real zeros,
counted with multiplicity.

The origin is already a zero of multiplicity at least two.  If $xw>yz$, it has
multiplicity exactly two and $F_A(t)>0$ for small positive $t$. 
 Since $F_A(t)\to+\infty$ as $t\to+\infty$, the function is
also positive for sufficiently large $t$.
Any positive zero would contribute at least two to the total
multiplicity of the positive zeros: a zero of even multiplicity
already contributes at least two, whereas a zero of odd multiplicity
changes the sign and therefore requires another positive zero,
since $F_A$ is positive both for sufficiently small positive $t$
and for sufficiently large $t$.
Together with the double zero at the origin, this would give
at least four real zeros counted with multiplicity, contradicting
Laguerre's rule of signs.
Thus $F_A(t)>0$ for all $t>0$.  If $xw=yz$, the zero at the origin has
multiplicity at least three, so there are no other real zeros; since
$F_A(1)>0$, the function is again positive on $(0,\infty)$.
\end{proof}

The following factorization gives another way to read the theorem:
\begin{equation}\label{eq:factorization}
 F_A(t)=(a^t-1)(d^t-1)-(b^t-1)(c^t-1).
\end{equation}
After dividing both products by $t^2$ and letting $t\downarrow0$, their
comparison becomes exactly condition \eqref{eq:classification}.

The proof of the $\TP_3$ critical-exponent theorem in
\cite{FallatJohnsonSokal2017} already contains the main zero-counting idea.  In
particular, it identifies the double zero at the origin, the coefficient
$\log a\log d-\log b\log c$, and the three sign changes of the exponential
polynomial.  The necessary-and-sufficient condition for all positive
exponents, however, does not seem to have been stated there.  Thus
Theorem~\ref{thm:classification} states this consequence as an explicit
classification; its proof is an extraction of the earlier zero-counting argument.

\begin{corollary}[Exact threshold in order three]\label{cor:threshold}
Let $A$ be as in Theorem~\ref{thm:classification}.
\begin{enumerate}[label=\textup{(\roman*)}]
 \item If $(\log a)(\log d)\geq(\log b)(\log c)$, then $\tau(A)=0$.
 \item If $(\log a)(\log d)<(\log b)(\log c)$, then $F_A$ has a unique
 positive zero $t_A\in(0,1)$, and
 \[
  F_A(t)<0\quad(0<t<t_A),\qquad
  F_A(t)>0\quad(t>t_A).
 \]
 In particular, $\tau(A)=t_A$.
\end{enumerate}
\end{corollary}

\begin{proof}
Part \textup{(i)} follows immediately from
Theorem~\ref{thm:classification}.  For part \textup{(ii)}, the function is
negative immediately to the right
of zero and positive at $t=1$, so it has a zero in $(0,1)$.  Laguerre's rule,
together with the double zero at the origin, shows that this is the only other
real zero and that it is simple.  The sign assertions and the formula for the
threshold follow.
\end{proof}

When the second case of Corollary~\ref{cor:threshold} holds, we call the
unique number $t_A\in(0,1)$ the \emph{transition exponent} of $A$.

Thus, in order three, the positivity set is always a single interval:
\[
 \mathcal P(A)=(0,\infty)
 \quad\text{or}\quad
 \mathcal P(A)=(t_A,\infty).
\]

Corollary~\ref{cor:threshold} gives a matrix-by-matrix description of the
positivity set for $3\times3$ totally positive matrices, including the
zero-threshold case.  It refines the zero-counting method in
\cite{FallatJohnsonSokal2017}.  The related threshold theorem in
\cite{AlseidiAlsaidiGarloff2026} concerns a different preservation with problem, namely the
Hurwitz stability of coefficientwise powers of degree-five polynomials.

\section{Quantitative bounds for the order-three threshold}

Assume throughout this section that $xw<yz$,
so that the unique transition exponent $t_A\in(0,1)$ exists.  Define
\begin{equation*}\label{eq:deltagap}
 \delta:=\log\frac{yz}{xw}>0,
 \qquad
 \Delta:=x+w-y-z>0,
 \qquad
 M:=x+w-\frac{y+z}{2}>0.
\end{equation*}
Consider the logarithmic determinant ratio
\begin{equation}\label{eq:H}
 H_A(t):=\log\frac{(e^{tx}-1)(e^{tw}-1)}
 {(e^{ty}-1)(e^{tz}-1)}.
\end{equation}
Then $H_A(t_A)=0$ and
$ \lim_{t\downarrow0}H_A(t)=-\delta$.
Indeed, by \eqref{eq:factorization}, the equality $F_A(t_A)=0$ is equivalent
to $H_A(t_A)=0$.  Moreover, using $e^{tu}-1=tu+O(t^2)$ as $t\downarrow0$,
we obtain
\[
 \lim_{t\downarrow0}H_A(t)=\log\frac{xw}{yz}=-\delta.
\]

The following lemma shows that the logarithmic determinant ratio is strictly
increasing and gives two-sided bounds for its derivative.

\begin{lemma}\label{lem:Hmonotone}
The function $H_A$ is strictly increasing on $(0,\infty)$.  More precisely,
\begin{equation}\label{eq:Hprimebounds}
 \frac{\Delta}{2}\leq H_A'(t)\leq M
 \qquad(t>0).
\end{equation}
\end{lemma}

\begin{proof}
Set $\psi_t(s):=\frac{s}{1-e^{-ts}},~s>0$.
Differentiating \eqref{eq:H} gives
$H_A'(t)=\psi_t(x)+\psi_t(w)-\psi_t(y)-\psi_t(z)$.
For $u>0$, the elementary inequalities
$ \frac1u+\frac12\leq\frac1{1-e^{-u}}\leq\frac1u+1$
imply
$\frac1t+\frac{s}{2}\leq\psi_t(s)\leq\frac1t+s$.
Applying the upper estimate to the positive terms and the lower estimate to
the negative terms yields $H_A'(t)\leq M$.
For the lower estimate, write
\[
 \psi_t(s)=\frac1t+\frac{s}{2}+k_t(s),
 \qquad
 k_t(s)=\frac1t\left\{\frac{ts}{2}\coth\frac{ts}{2}-1\right\}.
\]
The function $k_t$ is increasing and convex on $(0,\infty)$.  Indeed, for
$K(u)=\frac{u}{2}\coth\frac{u}{2}-1$, one has
\[
 K'(u)=\frac12\coth\frac{u}{2}
       -\frac{u}{4}\operatorname{csch}^2\frac{u}{2}>0
\]
and
\[
 K''(u)=\frac14\operatorname{csch}^2\frac{u}{2}
        \left(u\coth\frac{u}{2}-2\right)>0.
\]
Now $w>\max(y,z)$ and $x+w>y+z$, so the pair $(x,w)$ weakly
majorizes $(y,z)$ from above.  The standard characterization of weak
majorization of vectors by increasing convex functions
\cite[Chapter~3.A]{MarshallOlkinArnold2011} therefore gives
$ k_t(x)+k_t(w)\geq k_t(y)+k_t(z)$.
Consequently,
\[
 H_A'(t)\geq\frac{x+w-y-z}{2}=\frac{\Delta}{2}>0.
\]
\end{proof}

We now use the derivative estimates in Lemma~\ref{lem:Hmonotone} to obtain
explicit bounds for the transition exponent.

\begin{theorem}
\label{thm:thresholdbounds}
Let $A$ be a normalized $3\times3$ totally positive matrix satisfying
$xw<yz$.  Its unique transition exponent satisfies
\begin{equation}\label{eq:thresholdbounds}
 \frac{\delta}{M}\leq t_A
 \leq\min\left\{1,\frac{2\delta}{\Delta}\right\}.
\end{equation}
Equivalently,
\[
 \frac{\displaystyle\log\frac{\log b\,\log c}
 {\log a\,\log d}}
 {\displaystyle\log a+\log d-\frac{1}{2}(\log b+\log c)}
 \leq t_A
\]
and
\[
 t_A\leq
 \min\left\{1,
 \frac{\displaystyle2\log\frac{\log b\,\log c}
 {\log a\,\log d}}
 {\displaystyle\log a+\log d-\log b-\log c}
 \right\}.
\]
\end{theorem}

\begin{proof}
Since $H_A(t_A)=0$ and $H_A(0+)=-\delta$, integration gives
$ \delta=\int_0^{t_A}H_A'(s)\,ds$.
Using \eqref{eq:Hprimebounds}, we obtain
$\frac{\Delta}{2}t_A\leq\delta\leq Mt_A$.
This gives the two analytic bounds in \eqref{eq:thresholdbounds}; the additional
upper bound $t_A<1$ follows from Corollary~\ref{cor:threshold}.
\end{proof}

Lemma~\ref{lem:Hmonotone} also gives a stable numerical method for computing
$t_A$.  Applying bisection to $H_A$ on the interval in
\eqref{eq:thresholdbounds} converges to the unique root.  The bounds do not
change under positive diagonal scaling, because $a,b,c,d$ are fixed by the
normalized representative.

We next show that these thresholds are not only bounded: every value in the
allowed interval actually occurs.

\section{Exact realization of all order-three thresholds}

We now show that every number in $[0,1)$ is the Hadamard threshold of some
$3\times3$ totally positive matrix.  We construct a family of matrices
$A_s$, where $s\in(0,1)$, such that $\tau(A_s)=s$.  We first prove an
inequality that will be used to show that these matrices are totally positive.
Fix two numbers
%\begin{equation}\label{eq:xr}
$ 0<x<r$
%\end{equation}
and, for each $s\in(0,1)$, define
\begin{equation}\label{eq:ws}
 w_s:=\frac1s\log\left(
 1+\frac{(e^{sr}-1)^2}{e^{sx}-1}\right).
\end{equation}
Consider the symmetric normalized matrix
\begin{equation}\label{eq:As}
 A_s:=
 \begin{pmatrix}
  1&1&1\\
  1&e^x&e^r\\
  1&e^r&e^{w_s}
 \end{pmatrix}.
\end{equation}

\begin{lemma}\label{lem:ws}
For $0<x<r$ and $0<s<1$, it follows that
\begin{equation}\label{eq:wsgap}
 w_s>2r-x>r.
\end{equation}
\end{lemma}

\begin{proof}
Put $u=e^{sx}$ and $v=e^{sr}$, so $1<u<v$.  The elementary identity
\[
 u(v-1)^2-(u-1)(v^2-u)=(v-u)^2>0
\]
gives $ \frac{(v-1)^2}{u-1}>\frac{v^2}{u}-1$.
Applying \eqref{eq:ws}, gives
$ e^{sw_s}>e^{s(2r-x)}$,
and hence $w_s>2r-x>r$.
\end{proof}

\begin{theorem}[Realization theorem]\label{thm:realization}
For every $s\in(0,1)$, the matrix $A_s$ in \eqref{eq:As} is totally positive
and
\begin{equation}\label{eq:taus}
 \tau(A_s)=s.
\end{equation}
Consequently,
$\{\tau(A):A\text{ is a }3\times3\text{ totally positive matrix}\}
 =[0,1)$.

\end{theorem}

\begin{proof}
For the parameters of $A_s$, the logarithmic determinant ratio is
\[
 H_s(t)=\log\frac{(e^{tx}-1)(e^{tw_s}-1)}{(e^{tr}-1)^2}.
\]
The definition of $w_s$ gives $H_s(s)=0$.  By Lemma~\ref{lem:ws},
$w_s>r$ and$ x+w_s>2r$,
so Lemma~\ref{lem:Hmonotone} shows that $H_s$ is strictly increasing.  Since
$s<1$, it follows that $H_s(1)>0$.  The entries and all minors of order two of
$A_s$ are positive, while
$ \det A_s=(e^x-1)(e^{w_s}-1)-(e^r-1)^2>0$
because $H_s(1)>0$.  Thus $A_s$ is totally positive.

Strict monotonicity also gives $H_s(t)<0$ for $0<t<s$ and $H_s(t)>0$ for
$t>s$.  Hence $A_s^{\circ t}$ fails to be totally positive below $s$ and is
totally positive above $s$, proving \eqref{eq:taus}.

We have shown that every $s\in(0,1)$ occurs.  The value $0$ also occurs; for
example, choose $0<x<r<w$ with $x+w>2r$ and $xw\geq r^2$, and apply
Theorem~\ref{thm:classification}.  Finally, the order-three critical-exponent
theorem implies $\tau(A)<1$ for every totally positive $A$, so no other values
occur.
\end{proof}

The family $A_s$ also explains what happens at both endpoints.  As
$s\downarrow0$,
\begin{equation*}\label{eq:wlimit0}
 w_s\longrightarrow\frac{r^2}{x},
\end{equation*}
so $xw_s\to r^2$ and the matrices approach the boundary between eternal and
non-eternal total positivity.  As $s\uparrow1$,
\begin{equation*}\label{eq:wlimit1}
 w_s\longrightarrow
 \log\left(1+\frac{(e^r-1)^2}{e^x-1}\right),
\end{equation*}
and the limiting matrix has determinant zero.  Thus, along this family,
thresholds approaching $1$ are obtained while the matrices approach the
boundary of the totally positive cone.

Thus the threshold is not only contained in $[0,1)$; it takes every value in
this interval, already within a two-parameter family of symmetric normalized
matrices.

\section{Structural properties and stability of the threshold}

We first record how the threshold behaves under the standard transformations
of totally positive matrices.  The same statements hold for $\tau_r$.

\begin{proposition}[Transformation laws]\label{prop:transformations}
Let $A$ be a matrix with positive entries.
\begin{enumerate}[label=\textup{(\roman*)}]
 \item If $D_1,D_2$ are positive diagonal matrices of compatible sizes, then
 $ \tau(D_1AD_2)=\tau(A)$.

 \item Further,
  $\tau(A^{\mathsf T})=\tau(A)$.

 \item If $J_m,J_n$ are the reverse identity matrices, then
 $ \tau(J_mAJ_n)=\tau(A)$.

 \item For every $s>0$,
 \begin{equation}\label{eq:scalinglaw}
  \tau(A^{\circ s})=\frac{\tau(A)}{s},
 \end{equation}
 with the convention that $\infty/s=\infty$.
\end{enumerate}
\end{proposition}

\begin{proof}
For (i), use
$ (D_1AD_2)^{\circ t}=D_1^{\circ t}A^{\circ t}D_2^{\circ t}$.
Every minor on the right is multiplied by a positive row factor and a positive
column factor, so its sign is unchanged.  Part (ii) follows because the minors
of $A^{\mathsf T}$ are the transposes of the corresponding minors of $A$.
For (iii), reversing $k$ selected rows and $k$ selected columns contributes the
same sign $(-1)^{k(k-1)/2}$ twice, so every minor retains its sign.  Finally,
$ (A^{\circ s})^{\circ t}=A^{\circ st}$.
Thus all powers of $A^{\circ s}$ with $s\geq T$ are totally positive exactly when
all powers of $A$ with exponent $st>sT$ are totally positive, which proves
\eqref{eq:scalinglaw}.
\end{proof}

\begin{example}
Consider
\[
 A=\begin{pmatrix}
  1&1&1\\
  1&e&e^2\\
  1&e^2&e^5
 \end{pmatrix}.
\]
The inequalities in \eqref{eq:TPconditions} and \eqref{eq:detpositive} show
that $A$ is totally positive.  If $J_3$ is the reverse identity matrix, then
\[
 J_3AJ_3=
 \begin{pmatrix}
  e^5&e^2&1\\
  e^2&e&1\\
  1&1&1
 \end{pmatrix}.
\]
Proposition~\ref{prop:transformations}\textup{(iii)} gives
$\tau(J_3AJ_3)=\tau(A)$.  In fact, $A$ is eternally totally positive because
\[(\log e)(\log e^5)=5>4=(\log e^2)^2,\] so both thresholds are zero.
\end{example}

The next proposition shows that passing to a submatrix cannot increase the
Hadamard threshold.

\begin{proposition}\label{prop:submatrix}
If $B$ is a submatrix of $A$, then
$\tau_r(B)\leq\tau_r(A)$
whenever both sides are defined using the minors available in the respective
matrix sizes.  In particular, for a square submatrix $B$,
$\tau(B)\leq\tau(A)$.
\end{proposition}

\begin{proof}
Every minor of $B^{\circ t}$ is also a minor of $A^{\circ t}$.  Hence total
positivity of $A^{\circ t}$ implies total positivity of $B^{\circ t}$, and the
threshold inequality follows.
\end{proof}

We now prove continuity in order three.  Let $\mathcal D_3$ be the open subset
of $\mathbb R^4$ consisting of the parameter quadruples $(a,b,c,d)$ associated
with the normalized totally positive matrices in \eqref{eq:normalized}.  We
give $\mathcal D_3$ the topology induced from $\mathbb R^4$.

\begin{theorem}[Continuity in order three]\label{thm:continuity}
The map
$\tau:\mathcal D_3\longrightarrow[0,1)$
is continuous.
\end{theorem}

\begin{proof}
It is convenient to use logarithmic coordinates $(x,y,z,w)$.  Set
$Q(x,y,z,w):=xw-yz$.
In the open region $Q<0$, Corollary~\ref{cor:threshold} \textup{(ii)} shows
that the threshold is the unique positive solution of
$ H_A(t)=0$,
and Lemma~\ref{lem:Hmonotone} gives $\partial H_A/\partial t>0$.  The implicit
function theorem therefore shows that $t_A=\tau(A)$ depends continuously (in
fact, smoothly) on $(x,y,z,w)$ throughout this region.

In the region $Q>0$, Theorem~\ref{thm:classification} gives $\tau(A)=0$, so
continuity is immediate.  It remains to consider a point with $Q=0$.  If
$A_n\to A$ and $Q(A_n)<0$, then
\[
 0\leq\tau(A_n)\leq\frac{2\delta_n}{\Delta_n},
 \qquad
 \delta_n=\log\frac{y_nz_n}{x_nw_n}.
\]
Here $\delta_n\to0$, while
$\Delta_n=x_n+w_n-y_n-z_n
 \longrightarrow x+w-y-z>0$
because $A$ is totally positive.  Hence $\tau(A_n)\to0=\tau(A)$.  Sequences
eventually satisfying $Q(A_n)\geq0$ already have zero threshold.  This proves
continuity at the separating hypersurface and completes the proof.
\end{proof}

\begin{remark}
The theorem concerns the interior of the totally positive cone.  If the
definition of $\tau$ is retained for all matrices with positive entries, the
family in Theorem~\ref{thm:realization} has a boundary limit with threshold
$1$.  Thus the missing endpoint of the range $[0,1)$ is naturally attained on
the boundary where the determinant at $t=1$ vanishes.
\end{remark}

Hence the order-three threshold is unchanged by the natural equivalences,
satisfies the expected power-scaling rule, and varies continuously inside the
totally positive cone.

\section{Logarithmic cores and the higher-order reduction}

We next introduce a quantity that is unchanged by diagonal scaling and
controls a minor near $t=0$.  Let $B=(b_{ij})_{i,j=1}^k$ have positive
entries.  Its normalized form is
\begin{equation*}\label{eq:normalizedB}
 N(B)=(n_{ij})_{i,j=1}^k,
 \qquad
 n_{ij}:=\frac{b_{11}b_{ij}}{b_{i1}b_{1j}}.
\end{equation*}
The first row and first column of $N(B)$ consist of ones.  Define the
\emph{logarithmic core} of $B$ by
\begin{equation*}\label{eq:logcore}
 \Lambda(B):=
 \left(
 \log\frac{b_{11}b_{ij}}{b_{i1}b_{1j}}
 \right)_{i,j=2}^k.
\end{equation*}
This is a $(k-1)\times(k-1)$ matrix and is unchanged by positive diagonal
scaling of $B$.

The logarithmic core determines the leading behavior of a minor as the
Hadamard exponent approaches zero.

\begin{proposition}\label{prop:zeroexpansion}
For every $k\times k$ matrix $B$ with positive entries, there is a positive
function $c_B(t)$ such that
\begin{equation}\label{eq:minorExpansion}
 \det(B^{\circ t})
 =c_B(t)\bigl\{t^{k-1}\det\Lambda(B)+O(t^k)\bigr\}
 \qquad(t\downarrow0).
\end{equation}
In particular, the first potentially nonzero coefficient after normalization
is $\det\Lambda(B)$.
\end{proposition}

\begin{proof}
There exist positive diagonal matrices $D_1,D_2$ such that
$N(B)=D_1BD_2$.  Hence
\[
 \det(B^{\circ t})=c_B(t)\det(N(B)^{\circ t})
\]
for a positive scalar $c_B(t)$.  Subtract the first row of
$N(B)^{\circ t}$ from every other row and then expand along the first column.
This gives
\[
 \det(N(B)^{\circ t})
 =\det\bigl(n_{ij}^{t}-1\bigr)_{i,j=2}^k.
\]
Since
$n_{ij}^{t}-1=t\log n_{ij}+O(t^2)$,
the multilinearity of the determinant yields
\[
 \det(N(B)^{\circ t})
 =t^{k-1}\det\Lambda(B)+O(t^k),
\]
which proves \eqref{eq:minorExpansion}.
\end{proof}

\begin{corollary}[Infinitesimal obstruction]\label{cor:logobstruction}
Let $A$ be a matrix with positive entries.
\begin{enumerate}[label=\textup{(\roman*)}]
 \item If $A$ is eternally totally positive, then
 $\det\Lambda(B)\geq 0$,
 for every square submatrix $B$ of $A$.
 \item If $\det\Lambda(B)<0$ for some square submatrix $B$, then
 $A^{\circ t}$ is not totally positive for all sufficiently small $t>0$;
 consequently, $A$ is not eternally totally positive.
 \item If $\det\Lambda(B)>0$ for every square submatrix $B$ of order at least
 two, then $A^{\circ t}$ is totally positive for all sufficiently small
 $t>0$.
\end{enumerate}
\end{corollary}

\begin{proof}
Parts (i) and (ii) follow immediately from
Proposition~\ref{prop:zeroexpansion}.  For (iii), each minor is positive for
all sufficiently small positive $t$.  Since a finite matrix has only finitely
many minors, one common interval $(0,\varepsilon)$ works for each of them.
\end{proof}

For a $2\times2$ submatrix, the logarithmic core is the scalar logarithm of
its cross ratio, and its positivity is equivalent to positivity of the minor.
For the normalized $3\times3$ matrix \eqref{eq:normalized},
\[
 \Lambda(A)=
 \begin{pmatrix}x&y\\ z&w\end{pmatrix},
 \qquad
 \det\Lambda(A)=xw-yz.
\]
Thus, in order three, the local logarithmic-core condition is also sufficient
for every $t>0$.  In higher orders we use it only as a local necessary
condition; we do not claim that it is globally sufficient.

The threshold can also be computed minor by minor.  For a square matrix
$B$ with positive entries, put
\begin{equation*}\label{eq:betaB}
 \beta(B):=\sup\{t>0:\det(B^{\circ t})\leq0\},
\end{equation*}
where the supremum of the empty set is taken to be zero.  When $A$ is totally
positive, eventual total positivity makes all these quantities finite.

Since total positivity is determined by the signs of all minors, the
threshold of a matrix can be computed from the last nonpositive exponent of
each square submatrix.

\begin{proposition}\label{prop:minorthreshold}
If $A$ is a finite totally positive matrix, then
\begin{equation}\label{eq:taumax}
 \tau(A)=\max_{B}\beta(B),
\end{equation}
where the maximum ranges over all square submatrices $B$ of $A$.
More generally,
\[\tau_r(A)=\max_{\operatorname{ord}B\leq r}\beta(B).\]

\end{proposition}

\begin{proof}
The matrix $A^{\circ t}$ is totally positive exactly when
$\det(B^{\circ t})>0$ for every square submatrix $B$.  The union of the bad
sets of the finitely many minors therefore has supremum equal to the maximum
of their individual suprema, which is \eqref{eq:taumax}.
The formula for $\tau_r(A)$ follows in the same way after restricting the
argument to square submatrices of order at most $r$.
\end{proof}

\begin{corollary}[Exact reduction for $4\times4$ matrices]
\label{cor:fourbyfour}
Let $A$ be a $4\times4$ totally positive matrix.  Then
\begin{equation*}\label{eq:tau4}
 \tau(A)=\max\left\{
 \beta(A),
 \max_{\substack{B\subset A\\ \operatorname{ord}B=3}}\tau(B)
 \right\}.
\end{equation*}
Every threshold $\tau(B)$ in the second term is given by the order-three
classification and is strictly less than $1$.  Thus all genuinely new
order-four behavior is concentrated in the exponential polynomial
\begin{equation}\label{eq:det4}
 t\longmapsto\det(A^{\circ t}).
\end{equation}
For symmetric $A$, the known universal result gives $\tau(A)\leq2$; for
general $A$, no finite uniform bound exists.
\end{corollary}

\begin{proof}
Minors of orders one and two have zero threshold.  Apply
Proposition~\ref{prop:minorthreshold} and Theorem~\ref{thm:classification} to
the minors of order three.  The final assertions are the known critical
exponent results summarized in Section~\ref{sec:background}.
\end{proof}

The corollary reduces the order-four problem to a finite calculation.  We need
the last nonpositive zero of \eqref{eq:det4}, the local obstruction
$\det\Lambda(A)$, and the known thresholds of the $3\times3$ submatrices.

\section{The symmetric order-four determinant}

Let $A$ be a normalized symmetric $4\times4$ matrix, written in block form as
\begin{equation*}\label{eq:sym4block}
 A=\begin{pmatrix}
  1&\mathbf 1^{\mathsf T}\\
  \mathbf 1&\exp^{\circ}L
 \end{pmatrix},
 \qquad L=\Lambda(A)\in\mathbb R^{3\times3}_{\rm sym},
\end{equation*}
where $\exp^{\circ}L=(e^{\ell_{ij}})$ and $\mathbf 1$ is the three-dimensional
all-ones vector.  Subtracting the first row from the other rows gives
\begin{equation}\label{eq:sym4det}
 \det(A^{\circ t})
 =\det\bigl(\exp^{\circ}(tL)-\mathbf 1\mathbf 1^{\mathsf T}\bigr).
\end{equation}

The leading term in the zero-exponent expansion may vanish when $\det L=0$.
We therefore compute the next coefficient.

\begin{proposition}
Let $A$ be a normalized $4\times4$ matrix with positive entries and
logarithmic core $L$. Then
\begin{equation}\label{eq:fourthexpansion}
\det(A^{\circ t})
=t^3\det L
+\frac{t^4}{2}
\operatorname{tr}\bigl(\operatorname{adj}(L)L^{\circ2}\bigr)
+O(t^5),
\qquad(t\to0),
\end{equation}
where $L^{\circ2}$ denotes the Hadamard square and
$\operatorname{adj}(L)$ is the classical adjugate.
If $L$ is invertible, the coefficient of $t^4$ can equivalently
be written as $\frac{\det L}{2}\operatorname{tr}\bigl(L^{-1}L^{\circ2}\bigr)$.

\end{proposition}

\begin{proof}
Since $A$ is normalized, its first row and first column consist
entirely of ones. Write
\[
A=\begin{pmatrix}
1&\mathbf1^{\mathsf T}\\
\mathbf1&M
\end{pmatrix},
\]
where $\mathbf1\in\mathbb R^3$ is the all-ones vector and $M$ is
the lower-right $3\times3$ block. The logarithmic core is
\[
L=(\log M_{ij})_{i,j=1}^{3}.
\]
The lower-right block of $A^{\circ t}$ is
$(e^{tL_{ij}})_{i,j=1}^{3}$.
Subtracting the first row from each of the remaining three rows
preserves the determinant and makes the first column equal to
$(1,0,0,0)^{\mathsf T}$. Expanding along this column gives
$\det(A^{\circ t})=\det\bigl(e^{tL_{ij}}-1\bigr)_{i,j=1}^{3}$.
The entrywise Taylor expansion yields
\[
\bigl(e^{tL_{ij}}-1\bigr)_{i,j=1}^{3}
=t\left(L+\frac{t}{2}L^{\circ2}+O(t^2)\right).
\]
Factoring $t$ from each of the three rows, we obtain
\[
\det(A^{\circ t})
=t^3\det\left(L+\frac{t}{2}L^{\circ2}+O(t^2)\right).
\]
For fixed square matrices $L$ and $N$, the first variation formula
for the determinant is
\[
\det(L+\varepsilon N)
=\det L
+\varepsilon\operatorname{tr}
 \bigl(\operatorname{adj}(L)N\bigr)
+O(\varepsilon^2).
\]
This formula holds even when $L$ is singular. An additional
matrix perturbation of order $O(t^2)$ changes the determinant
by $O(t^2)$, since the determinant is a polynomial in its entries.
Taking $\varepsilon=t$ and $N=\frac12L^{\circ2}$ therefore gives
\[
\det\left(L+\frac{t}{2}L^{\circ2}+O(t^2)\right)
=\det L
+\frac{t}{2}\operatorname{tr}
 \bigl(\operatorname{adj}(L)L^{\circ2}\bigr)
+O(t^2).
\]
Multiplying by $t^3$ proves the asserted expansion.
Finally, if $L$ is invertible, use
$\operatorname{adj}(L)=(\det L)L^{-1}$.
\end{proof}

\begin{corollary}[The positive-semidefinite boundary]\label{cor:psdboundary}
If $L$ is positive semidefinite and $\det L=0$, then the coefficient of $t^4$
in \eqref{eq:fourthexpansion} is nonnegative.  More precisely,
$ \operatorname{tr}\bigl(\operatorname{adj}(L)L^{\circ2}\bigr)\geq0$.
\end{corollary}

\begin{proof}
The adjugate of a $3\times3$ positive-semidefinite matrix with zero determinant
is positive semidefinite: it has rank one and is positive semidefinite when
$\operatorname{rank}L=2$, and it is zero when $\operatorname{rank}L\leq1$.
The Schur product theorem \cite{Schur1911} gives $L^{\circ2}\geq0$.  Since
$\operatorname{tr}(PQ)\geq0$ for positive-semidefinite matrices $P,Q$, the
claim follows.
\end{proof}

The preceding expansion is local, whereas the next theorem 
shows that positive definiteness of the logarithmic core guarantees 
positivity of the determinant of entire matrix for every positive Hadamard power.
\begin{theorem}
\label{thm:logcorePD}
Let $A$ be a normalized symmetric $4\times4$ matrix with positive entries.  If
$L=\Lambda(A)$ is positive definite, then
\begin{equation}\label{eq:fullpositive}
 \det(A^{\circ t})>0\qquad\text{for every }t>0.
\end{equation}
If, in addition, every $3\times3$ submatrix of $A$ is eternally totally
positive and every $2\times2$ minor of $A$ is positive, then $A$ is eternally
totally positive.
\end{theorem}

\begin{proof}
The entrywise exponential has the convergent expansion
\[
 \exp^{\circ}(tL)-\mathbf 1\mathbf 1^{\mathsf T}
 =\sum_{m=1}^{\infty}\frac{t^m}{m!}L^{\circ m}.
\]
The first term $tL$ is positive definite.  Every remaining term is positive
semidefinite by repeated application of the Schur product theorem
\cite{Schur1911}.  Therefore
the sum is positive definite for every $t>0$, and
\eqref{eq:sym4det} proves \eqref{eq:fullpositive}.

Under the additional hypotheses, all minors of orders two and three of every
$A^{\circ t}$ are positive.  Positivity of the entries is automatic, and
\eqref{eq:fullpositive} handles the only minor of order four.  Hence
$A^{\circ t}$ is totally positive for every $t>0$.
\end{proof}

\begin{corollary}[Localization of thresholds above one]
\label{cor:aboveone}
Let $A$ be a normalized symmetric $4\times4$ totally positive matrix.  If its
logarithmic core $L$ is positive definite, then
\begin{equation}\label{eq:tau3control}
 \tau(A)=
 \max_{\substack{B\subset A\\\operatorname{ord}B=3}}\tau(B)<1.
\end{equation}
Consequently,
$\tau(A)>1$ inplies $\Lambda(A)$ is not positive definite.
\end{corollary}

\begin{proof}
Theorem~\ref{thm:logcorePD} gives $\beta(A)=0$ for the full determinant.
Equation \eqref{eq:tau3control} follows from Corollary~\ref{cor:fourbyfour}, and
each order-three threshold is strictly below $1$.  The final assertion is the
contrapositive.
\end{proof}

We now show that thresholds can be arbitrarily close to the universal endpoint
$2$.

\begin{theorem}
\label{thm:neartwo}
For every $s\in(1,2)$, there exists a symmetric Hankel $4\times4$ totally
positive matrix $A$ such that $ s<\tau(A)<2$.
Consequently,
$\sup\{\tau(A):A=A^{\mathsf T}\in\TP\}=2$.
\end{theorem}
\begin{proof}
The lower bound $\tau(A)>s$ follows from \cite[Example~5.5]{FallatJohnsonSokal2021}, while
the upper bound $\tau(A)<2$ follows from 
\cite[Proposition~5.6]{FallatJohnsonSokal2021}.
\end{proof}

The preceding theorem gives thresholds arbitrarily close to $2$, but it does
not prescribe their exact values.  The following kernel expansion allows us
to place the last zero exactly.  For distinct positive numbers
$u_1<\cdots<u_4$, an analytic function $g(z)=\sum_{k\geq0}c_kz^k$, and
$\varepsilon\downarrow0$, applying the Cauchy--Binet identity gives
\begin{equation}\label{eq:kernel-vandermonde}
 \det\bigl(g(\varepsilon u_i u_j)\bigr)_{i,j=1}^4
 =\varepsilon^6\left(\prod_{k=0}^3c_k\right)
  \prod_{1\leq i<j\leq4}(u_j-u_i)^2+O(\varepsilon^7).
\end{equation}
The remainder is locally uniform when $g$ varies analytically in finitely many
parameters.  This follows by writing the kernel matrix as
$V\operatorname{diag}(c_k\varepsilon^k)V^{\mathsf T}$ and observing that the
unique least-degree term in the Cauchy--Binet identity uses the columns $0,1,2,3$.

The preceding result only gives thresholds arbitrarily close to $2$.  We now
realize every value in $(1,2)$ and determine the exact range in the symmetric
$4\times4$ case.

\begin{theorem}
\label{thm:exactrange4}
For every $s\in(1,2)$, there exists a symmetric Hankel $4\times4$ totally
positive matrix $A$ satisfying $\tau(A)=s$.  Moreover,
\begin{equation}\label{eq:exactrange4}
 \{\tau(A):A=A^{\mathsf T}\in\TP\}
 =[0,1)\,\cup\,(1,2).
\end{equation}
In particular, neither $1$ nor $2$ is attained.
\end{theorem}

\begin{proof}
Fix $s\in(1,2)$ and choose $u_i=a^{i-1}$ with $a>1$.  For
$p>2$, $\delta>0$, and $\varepsilon>0$, set
\begin{equation*}\label{eq:exact-family}
 M_\varepsilon=(1+\varepsilon u_i u_j)_{i,j=1}^4,
 \qquad A_{\varepsilon,\delta}
 =M_\varepsilon+\delta M_\varepsilon^{\circ p}.
\end{equation*}
Both summands are Hankel.  The first is totally nonnegative, whereas the second
is totally positive for all sufficiently small $\varepsilon>0$.  To verify the
latter assertion without applying a strict power theorem to the rank-two
matrix $M_\varepsilon$, consider any contiguous $k\times k$ submatrix,
$1\leq k\leq4$.  After absorbing a positive power of $a$ into $\varepsilon$,
its entries have the form
\[
 (1+\eta v_i v_j)^p,\qquad v_i=a^{i-1},\quad \eta>0.
\]
The $k\times k$ version of \eqref{eq:kernel-vandermonde} has leading
coefficient
\[
 \eta^{k(k-1)/2}
 \left(\prod_{m=0}^{k-1}\binom{p}{m}\right)
 \prod_{1\leq i<j\leq k}(v_j-v_i)^2.
\]
It is positive for $p>2$ and $k\leq4$.  There are only finitely many
contiguous minors, so they are all positive when $\varepsilon$ is sufficiently
small; Fekete's criterion then gives
$M_\varepsilon^{\circ p}\in\TP_4$.  Since
$M_\varepsilon\in\TN_4$, the corrected Hankel sum theorem
\cite[Corollary~3.3(b)]{FallatJohnsonSokal2017}, with the proof corrected in
\cite{FallatJohnsonSokal2021}, now yields
$A_{\varepsilon,\delta}\in\TP_4$.

Apply \eqref{eq:kernel-vandermonde} to
$g_{t,p,\delta}(z)=\bigl[(1+z)+\delta(1+z)^p\bigr]^t$.
Write 
\[h_0=1+\delta, \quad h_1=1+p\delta,\quad
h_2=\delta p(p-1),\quad\mbox{ and}\quad h_3=\delta p(p-1)(p-2),\]
 and set
\[
 R=\frac{h_0h_2}{h_1^2},\qquad S=\frac{h_0^2h_3}{h_1^3}.
\]
The only sign-bearing factors among the first four Taylor coefficients for
$t>1$ are
$ Q(t)=t-1+R$ and
 $P(t)=(t-1)(t-2)+3R(t-1)+S$.
Indeed,
\[
 c_2=\frac{t}{2}h_0^{t-2}h_1^2Q(t),\qquad
 c_3=\frac{t}{6}h_0^{t-3}h_1^3P(t).
\]
Thus $Q(t)>0$ for $t>1$, and the leading sign in
\eqref{eq:kernel-vandermonde} is the sign of $P(t)$.

First allow $p=2$.  Then $S=0$ and
\[
 R=\frac{2\delta(1+\delta)}{(1+2\delta)^2}.
\]
This strictly increases from $0$ to $1/2$.  Choose the unique $\delta_0>0$
such that $R=(2-s)/3$.  At this point $P(t)=(t-1)(t-s)$, so $s$ is a simple
larger root.  By the implicit-function theorem, for every $p>2$ sufficiently
close to $2$ there is a positive $\delta_p$ for which $P(s)=0$; the other root
remains below $s$ and $P'(s)>0$.

Normalize the determinant in \eqref{eq:kernel-vandermonde} by
$\varepsilon^6$.  Its analytic extension to $\varepsilon=0$ has nonzero
$\delta$-derivative at $(0,\delta_p,s)$, since the root condition is
transverse (at $p=2$ this derivative is
$3(s-1)\,dR/d\delta>0$).  A second application of the implicit-function
theorem gives
$\delta(\varepsilon)=\delta_p+O(\varepsilon),~
 \det\bigl(A_{\varepsilon,\delta(\varepsilon)}^{\circ s}\bigr)=0$.
For completeness, let
\[
 D_\varepsilon(t)=\varepsilon^{-6}
 \det\bigl(A_{\varepsilon,\delta(\varepsilon)}^{\circ t}\bigr).
\]
On a fixed compact neighbourhood of $[s,2]$, the functions
$D_\varepsilon$ and their first $t$-derivatives converge uniformly to a
positive continuous factor times $Q(t)P(t)$.  Since $P(s)=0$ and $P'(s)>0$,
there are $\eta,c>0$ such that $D_\varepsilon'(t)>c$ on
$[s,s+\eta]$ for all small $\varepsilon$.  As $D_\varepsilon(s)=0$, this
proves positivity on $(s,s+\eta]$ and negativity immediately to the left of
$s$.  On the remaining compact interval $[s+\eta,2]$, the limiting function
has a positive minimum because $Q(t)P(t)>0$; uniform convergence therefore
gives $D_\varepsilon(t)>0$ there as well.  This proves, rather than merely
asserts, that no additional full-determinant zero occurs in $(s,2]$.

Theorem~\ref{thm:FJScritical} \textup{(ii)} gives total positivity for
$t\geq2$, while Theorem~\ref{thm:FJcritical} applied to the proper submatrices
shows that every proper minor is positive for $t>1$.  Hence
$A_{\varepsilon,\delta(\varepsilon)}^{\circ t}$ is totally positive for every
$t>s$ and fails to be so immediately below $s$.  Its threshold is exactly
$s$.

Fix one such matrix with threshold $T\in(1,2)$.  Given $r\in(0,1)$, take
$q=T/r>T$ and set $B=A^{\circ q}$.  Then $B$ is symmetric and totally
positive, and Proposition~\ref{prop:transformations} gives
$\tau(B)=T/q=r$.  The value $0$ is attained by
$(e^{x_i x_j})_{i,j=1}^4$ with $0<x_1<\cdots<x_4$, since the exponential
kernel is strictly totally positive and every Hadamard power has the same form.

Finally, continuity of the finitely many minors implies that $t=1$ has a
neighbourhood on which every fixed totally positive matrix remains totally
positive; hence no threshold equals $1$.  In the symmetric order-four class,
total positivity holds for all $t\geq2$ and strict positivity at $t=2$
persists on a left neighbourhood; hence no threshold equals $2$.  This proves
\eqref{eq:exactrange4}.
\end{proof}

Theorem~\ref{thm:neartwo} applies the sharpness argument of
\cite{FallatJohnsonSokal2017} to the threshold of a fixed matrix.
Theorem~\ref{thm:exactrange4} goes further: its two-parameter perturbation gives
each threshold exactly.  By Corollary~\ref{cor:aboveone}, all examples with
threshold above $1$ have an indefinite or singular logarithmic core.

\section{Eternal total positivity of symmetric Toeplitz matrices}
\label{sec:toeplitz}

We now apply the order-three classification to a structured family in higher
orders. The tools concerning positive definiteness are classical: infinite
divisibility is characterized by conditional positive semidefiniteness of the
entrywise logarithm; see \cite{Bhatia2006}. Total positivity requires additional
control of nonprincipal minors. The use of positive-definiteness results
alongside lower-order minors also occurs in
\cite{FallatJohnsonSokal2017}. Our purpose here is to give an explicit
order-four classification and to identify the additional obstruction in order
five, rather than to introduce a new criterion regarding infinite divisibility.
The Toeplitz-preserver results in \cite{BeltonEtAl2022,BeltonEtAl2023}
concern transformations that work to an entire class; here we determine
whether all positive powers of an individual matrix are totally positive.

\subsection{A recursive criterion in every order}

A real symmetric $n\times n$ matrix $M$ is said to be
\emph{conditionally positive semidefinite} if $v^{\mathsf T}Mv\geq0$
for every $v\in\R^n$ satisfying
$\mathbf1^{\mathsf T}v=\sum_{i=1}^n v_i=0$.
Here and below $\mathbf1$ denotes the all-ones vector of the relevant size.
A square submatrix is \emph{contiguous} if its rows and columns are both
consecutive. We use the classical criterion that a matrix is strictly totally
positive if all its contiguous square minors are positive; see
\cite{FallatJohnson2011, Pinkus2010}.

\begin{lemma}\label{lem:toeplitz-strict-ID}
Let $A$ be a real symmetric $n\times n$ matrix with positive entries such that
$a_{ii}a_{jj}>a_{ij}^2$ for $i\ne j$. Then the following are equivalent:
\begin{enumerate}[label=\textup{(\roman*)}]
\item $\log^{\circ}A$ is conditionally positive semidefinite;
\item $\Lambda(A)\succeq0$;
\item $A^{\circ t}$ is positive definite for every $t>0$.
\end{enumerate}
\end{lemma}

\begin{proof}
Set $M=(m_{ij})=\log^{\circ}A$ and write
\[
M=\begin{pmatrix}m_{11}&r^{\mathsf T}\\r&C\end{pmatrix}.
\]
Then
\[\Lambda(A)=C-r\mathbf1^{\mathsf T}-\mathbf1r^{\mathsf T}
+m_{11}\mathbf1\mathbf1^{\mathsf T},\]
where $\mathbf1\in\R^{n-1}$. For $u\in\R^{n-1}$, take
$v=(-\mathbf1^{\mathsf T}u,u^{\mathsf T})^{\mathsf T}\in\R^n$.
This vector satisfies $\sum_i v_i=0$, and symmetry of $M$ gives
$v^{\mathsf T}Mv=u^{\mathsf T}Cu
-2(\mathbf1^{\mathsf T}u)(r^{\mathsf T}u)
+m_{11}(\mathbf1^{\mathsf T}u)^2
=u^{\mathsf T}\Lambda(A)u$.
Every vector in $\R^n$ with zero sum has this form uniquely,
which proves (i)$\iff$(ii).
For (iii)$\Rightarrow$(i), the entrywise Taylor expansion gives
$A^{\circ t}=\mathbf1\mathbf1^{\mathsf T}+tM+O(t^2)$.
For a fixed $v\in\R^n$ with $\mathbf1^{\mathsf T}v=0$, the constant
term has quadratic form $(\mathbf1^{\mathsf T}v)^2=0$. Hence, if (iii) holds,
\[
0\leq\lim_{t\downarrow0}t^{-1}v^{\mathsf T}A^{\circ t}v
=v^{\mathsf T}Mv,
\]
so (i) follows.

Suppose (ii) holds. The spectral theorem gives
$\Lambda(A)=QDQ^{\mathsf T}$ with $Q$ orthogonal and $D$ diagonal
with nonnegative entries. Set $S=D^{1/2}Q^{\mathsf T}$, so that
$\Lambda(A)=S^{\mathsf T}S$. Thus the columns $p_2,\ldots,p_n$ of $S$
realize $\Lambda(A)$ as a Gram matrix, even when it is singular.
Put $p_1=0$. With $q_i=m_{i1}-m_{11}/2$, we have
$m_{ij}=q_i+q_j+\langle p_i,p_j\rangle$ and $A^{\circ t}=D_tK_tD_t$,
where $D_t=\operatorname{diag}(e^{tq_i})$ and
$K_t=(e^{t\langle p_i,p_j\rangle})$. Moreover,
\[
\|p_i-p_j\|^2=m_{ii}+m_{jj}-2m_{ij}
=\log\frac{a_{ii}a_{jj}}{a_{ij}^2}>0.
\]
Thus the $p_i$ are distinct. For any $z\in\R^n$,
\[
z^{\mathsf T}K_tz
=\sum_{m=0}^{\infty}\frac{t^m}{m!}
\left\|\sum_i z_i p_i^{\otimes m}\right\|^2,
\]
Here $p_i^{\otimes m}$ denotes the $m$-fold tensor product of $p_i$,
with $p_i^{\otimes0}=1$. The natural tensor-product inner product satisfies
$\langle p_i^{\otimes m},p_j^{\otimes m}\rangle
=\langle p_i,p_j\rangle^m$,
so the displayed identity follows by expanding the exponential.
The term for $m=0$ is $(\sum_i z_i)^2$.
If this expression is zero, then $\sum_i z_iP(p_i)=0$ for every polynomial
$P$. For each $j$, the polynomial
\[
P_j(x)=\prod_{i\ne j}
\frac{\langle x-p_i,p_j-p_i\rangle}{\|p_j-p_i\|^2}
\]
is well defined because the $p_i$ are distinct. It is one at $p_j$
and zero at the other points. Hence $z_j=0$ for every $j$.
This proves $K_t\succ0$ and therefore (iii), including when
$\Lambda(A)$ is singular.
\end{proof}

\begin{theorem}[Recursive symmetric criterion]\label{thm:recursive-ETP}
Let $n\geq3$ and let $A$ be a symmetric $n\times n$ totally positive matrix.
Then $A\in\ETP$ if and only if $\Lambda(A)\succeq0$ and each contiguous
$(n-1)\times(n-1)$ submatrix of $A$ belongs to $\ETP$.
\end{theorem}

\begin{proof}
If $A\in\ETP$, then for every $t>0$ the symmetric matrix
$A^{\circ t}$ has positive leading principal minors. By Sylvester's
criterion it is positive definite, so
Lemma~\ref{lem:toeplitz-strict-ID} gives $\Lambda(A)\succeq0$.
Its submatrices are ETP as well. Conversely, that lemma gives
$\det(A^{\circ t})>0$ for every $t>0$. Each smaller contiguous square
submatrix is contained in a contiguous $(n-1)\times(n-1)$ submatrix,
so its determinant is positive by hypothesis. The contiguous-minor criterion
then gives $A^{\circ t}\in\TP$ for every $t>0$.
\end{proof}

Consider the normalized symmetric Toeplitz matrix
\begin{equation}\label{eq:toeplitz-n}
T_n=(e^{-r_{|i-j|}})_{i,j=1}^n,\qquad r_0=0,
\end{equation}
and set
\[
G_n=(r_i+r_j-r_{|i-j|})_{i,j=1}^{n-1},\qquad
B_{n-1}=T_n[\{1,\ldots,n-1\};\{2,\ldots,n\}].
\]
Here $G_n=\Lambda(T_n)$. Its four contiguous blocks of order $n-1$ are
$T_{n-1},T_{n-1},B_{n-1},B_{n-1}^{\mathsf T}$. Consequently:

\begin{corollary}\label{cor:toeplitz-recursion}
If $T_n$ in \eqref{eq:toeplitz-n} is totally positive, then
\[
T_n\in\ETP\quad\Longleftrightarrow\quad
G_n\succeq0,\quad T_{n-1}\in\ETP,\quad B_{n-1}\in\ETP.
\]
\end{corollary}

This is an exact recursive criterion, not an explicit classification in every
order: $B_{n-1}$ need not be symmetric, so the same symmetric criterion
cannot simply be iterated on that block.

\subsection{Explicit classification in order four}

\begin{theorem}[Symmetric Toeplitz order-four classification]
\label{thm:toeplitz4}
Let
\[
T_4=\begin{pmatrix}
1&a&b&c\\a&1&a&b\\b&a&1&a\\c&b&a&1
\end{pmatrix}
\]
be totally positive, and put $\alpha=-\log a$, $\beta=-\log b$,
$\gamma=-\log c$. Then
\begin{equation}\label{eq:toeplitz4-classification}
T_4\in\ETP\quad\Longleftrightarrow\quad
\beta\leq4\alpha,\qquad
\frac{(\beta-\alpha)^2}{\alpha}\leq\gamma
\leq\alpha+\frac{\beta^2}{2\alpha}.
\end{equation}
\end{theorem}

\begin{proof}
Total positivity implies
$\alpha>0$, $\beta>2\alpha$, and $\gamma>2\beta-\alpha$.
The two distinct contiguous blocks of order three are
\[
T_3=\begin{pmatrix}1&a&b\\a&1&a\\b&a&1\end{pmatrix},
\qquad
B_3=\begin{pmatrix}a&b&c\\1&a&b\\a&1&a\end{pmatrix}.
\]
Their logarithmic cores are, respectively,
\[
\begin{pmatrix}2\alpha&\beta\\\beta&2\beta\end{pmatrix},
\qquad
\begin{pmatrix}
\beta-2\alpha&\gamma-\alpha-\beta\\
\beta&\gamma-\alpha
\end{pmatrix}.
\]
Theorem~\ref{thm:classification} therefore gives
\[
T_3\in\ETP\iff\beta\leq4\alpha,
\qquad
B_3\in\ETP\iff\gamma\leq\alpha+\frac{\beta^2}{2\alpha}.
\]

To evaluate conditional positivity, let
$R=(r_{|i-j|})_{i,j=1}^4$, where $(r_0,r_1,r_2,r_3)=(0,\alpha,\beta,\gamma)$.
Since $\log^{\circ}T_4=-R$, we need $v^{\mathsf T}Rv\leq0$ on the
zero-sum subspace. On the antisymmetric vectors $(u,v,-v,-u)$,
\[
-\tfrac12 (u,v,-v,-u)R(u,v,-v,-u)^{\mathsf T}
=\gamma u^2+2(\beta-\alpha)uv+\alpha v^2.
\]
On the symmetric zero-sum vectors $(u,-u,-u,u)$ the quadratic form of $R$
is $2(\gamma-\alpha-2\beta)u^2$. To separate these quadratic forms,
let $J$ be the reversal matrix. Since $JR=RJ$, a symmetric vector $s$
and an antisymmetric vector $a$ satisfy
\[
s^{\mathsf T}Ra=(Js)^{\mathsf T}R(Ja)=-s^{\mathsf T}Ra=0.
\]
Every zero-sum vector is uniquely a sum of an antisymmetric vector and
a symmetric zero-sum vector. Thus conditional positivity is equivalent to
$\alpha\gamma\geq(\beta-\alpha)^2$ and
$ \gamma\leq\alpha+2\beta$.
If $\beta\leq4\alpha$, the upper bound from $B_3$ implies the last inequality,
because $\beta^2/(2\alpha)\leq2\beta$.
Corollary~\ref{cor:toeplitz-recursion} now proves the equivalence, with the
nonstrict inequalities included.
\end{proof}

\begin{example}[Gaussian boundary in every order]\label{ex:toeplitz-gaussian}
For $h>0$, the matrices $(e^{-h(i-j)^2})_{i,j=1}^n$ are ETP for every $n$.
Indeed, their positive Hadamard powers are positive diagonal scalings of
$(e^{2htij})_{i,j=1}^n$, which is strictly totally positive by the generalized
Vandermonde property. In order four this gives
$(\alpha,\beta,\gamma)=(h,4h,9h)$, on the boundary of
\eqref{eq:toeplitz4-classification}; the logarithmic core is singular.
\end{example}

\subsection{The additional obstruction in order five}

Write $T_5$ as in \eqref{eq:toeplitz-n}, with
$(r_1,r_2,r_3,r_4)=(\alpha,\beta,\gamma,\eta)$, and assume that $T_5$ is TP.
The remaining nonprincipal block is
\begin{equation*}\label{eq:toeplitz-B4}
B_4=\begin{pmatrix}
a&b&c&d\\1&a&b&c\\a&1&a&b\\b&a&1&a
\end{pmatrix},\qquad
(a,b,c,d)=(e^{-\alpha},e^{-\beta},e^{-\gamma},e^{-\eta}).
\end{equation*}

\begin{proposition}\label{prop:toeplitz5-CPD}
The condition $G_5\succeq0$ is equivalent to
\begin{equation}\label{eq:toeplitz5-blocks}
\begin{pmatrix}\eta&\gamma-\alpha\\\gamma-\alpha&\beta\end{pmatrix}\succeq0,
\qquad
\begin{pmatrix}
4\beta-\eta&\alpha+2\beta-\gamma\\
\alpha+2\beta-\gamma&4\alpha-\beta
\end{pmatrix}\succeq0.
\end{equation}
If $\beta<4\alpha$, then these conditions are equivalent to
\[
\frac{(\gamma-\alpha)^2}{\beta}\leq\eta
\leq4\beta-\frac{(\alpha+2\beta-\gamma)^2}{4\alpha-\beta}.
\]
If $T_4\in\ETP$ and $\beta=4\alpha$, then they force
$\gamma=9\alpha$ and $\eta=16\alpha$.
\end{proposition}

\begin{proof}
Let $R=(r_{|i-j|})_{i,j=1}^5$. Its quadratic form splits on the zero-sum
subspace into antisymmetric vectors $(u,v,0,-v,-u)$ and symmetric vectors
$(u,v,-2u-2v,v,u)$. On these two subspaces, $-\frac12v^{\mathsf T}Rv$
has the two matrices in \eqref{eq:toeplitz5-blocks}, respectively.
The scalar inequalities follow from the $2\times2$ semidefinite criterion.
For $\beta=4\alpha$, Theorem~\ref{thm:toeplitz4} gives $\gamma=9\alpha$;
the two blocks then require both $\eta\geq16\alpha$ and
$\eta\leq16\alpha$.
\end{proof}

Set $f(t)=\det(B_4^{\circ t})$. Direct expansion gives
\begin{equation}\label{eq:toeplitz5-f}
f(t)=P(e^{-\alpha t},e^{-\beta t},e^{-\gamma t})
-e^{-\eta t}D_3(e^{-\alpha t},e^{-\beta t}),
\end{equation}
where $D_3(a,b)=1-2a^2+2a^2b-b^2$
and
\begin{align*}
P(a,b,c)={}a^4-3a^2b+2a^2b^2+b^2-b^4
+2ac(1-b+b^2-a^2)+(a^2-b)c^2.
\end{align*}
The logarithmic core of $B_4$ is
\begin{equation*}\label{eq:toeplitz5-LB}
L_B=\begin{pmatrix}
\beta-2\alpha&\gamma-\beta-\alpha&\eta-\gamma-\alpha\\
\beta&\gamma-\alpha&\eta-\beta\\
2\beta-2\alpha&\beta+\gamma-\alpha&\beta+\eta-2\alpha
\end{pmatrix}.
\end{equation*}
For a positive $k\times k$ matrix $B=(b_{ij})$, the normalization
$n_{ij}=b_{ij}b_{11}/(b_{i1}b_{1j})$ gives the exact determinant factor
\[
c_B(t)=\left(
\frac{\prod_{i=1}^k b_{i1}\prod_{j=1}^k b_{1j}}{b_{11}^{\,k}}
\right)^t.
\]
For $B_4$, the first-column and first-row products are $a^2b$ and
$abcd$, respectively. Hence the normalization factor in
Proposition~\ref{prop:zeroexpansion} is
\[
c_{B_4}(t)=\left(\frac{b^2cd}{a}\right)^t
=e^{(\alpha-2\beta-\gamma-\eta)t},
\]
and consequently
$f(t)=e^{(\alpha-2\beta-\gamma-\eta)t}\{t^3\det L_B+O(t^4)\}$.
Since $c_{B_4}(t)=1+O(t)$, the leading coefficient remains $\det L_B$.
Although $\det L_B\geq0$ is necessary, it is not sufficient even when the
principal block is ETP and $G_5\succeq0$.

The contiguous order-three block not contained in $T_4$ is
\[
C_* = T_5[\{1,2,3\};\{3,4,5\}]
=\begin{pmatrix}b&c&d\\a&b&c\\1&a&b\end{pmatrix}.
\]
Its logarithmic core is
\[
\Lambda(C_*)=
\begin{pmatrix}
\alpha+\gamma-2\beta&\alpha+\eta-\beta-\gamma\\
\gamma-\alpha-\beta&\eta-2\beta
\end{pmatrix}.
\]
Since $T_5\in\TP$, this block is TP. Theorem~\ref{thm:classification}
therefore gives
\[
C_*\in\ETP\quad\Longleftrightarrow\quad
\det\Lambda(C_*)\geq0,
\]
that is,
$(\alpha+\gamma-2\beta)(\eta-2\beta)
\geq(\alpha+\eta-\beta-\gamma)(\gamma-\alpha-\beta)$.
This condition must be checked separately from $T_4\in\ETP$.

\begin{example}[Insufficiency of principal and leading-coefficient tests]
\label{ex:toeplitz5-counterexample}
Take
\[
(\alpha,\beta,\gamma,\eta)=(1,383/100,208/25,358/25).
\]
The resulting $T_5$ is TP. To give reproducible strict bounds, the minimum
of its contiguous minors of each order exceeds the following rational number:
\[
\begin{array}{c|ccccc}
\text{order}&1&2&3&4&5\\\hline
\text{lower bound}&6\cdot10^{-7}&4.6\cdot10^{-8}&6.4\cdot10^{-6}&1/100&13/25
\end{array}
\]
These bounds can be verified using rational arithmetic as follows. For each
$x=r_{|i-j|}$ replace $e^{-x}$ by
$S_{180}(x)=\sum_{k=0}^{180}(-x)^k/k!$. The alternating tail, whose terms
are decreasing from this index onward, has absolute value at most
$x^{181}/181!$. If $\varepsilon$ is the largest entry error, the determinant
error for an order-$k$ block is at most
$k!\,k\,2^{k-1}\varepsilon$, by telescoping the products in the Leibniz
formula; both the exact and approximating entries have absolute value below
$2$. This proves each displayed bound by rational comparisons.
The contiguous-minor criterion proves $T_5\in\TP$.

The conditions for $T_4\in\ETP$ hold since
\[
\beta<4,\qquad (\beta-1)^2=8.0089<8.32<8.33445=1+\beta^2/2.
\]
The diagonal entries of both blocks in \eqref{eq:toeplitz5-blocks} are positive,
and their determinants are $1.2632$ and $0.0544$. Thus $G_5\succeq0$.
Moreover, $\det L_B=799/250000>0$. Nevertheless,
\[
-1.963\cdot10^{-6}<f(1/25)<-1.962\cdot10^{-6}.
\]
The same rational-tail procedure, now with $x=r_{|i-j|}/25$, certifies this
interval. Therefore $T_5\notin\ETP$ despite all three conditions above.
For the same parameters, however,
\[
\Lambda(C_*)=\begin{pmatrix}1.66&3.17\\3.49&6.66\end{pmatrix},
\qquad
\det\Lambda(C_*)=-\frac{77}{10000}<0.
\]
Thus $C_*\notin\ETP$ as well: its determinant is negative for all
sufficiently small positive powers, by Proposition~\ref{prop:zeroexpansion}.
The example proves insufficiency of the three stated conditions; it does
not establish an order-four obstruction after all order-three blocks
have been checked to be ETP.
\end{example}

\subsection{A finite critical-point test and an open problem}

\begin{proposition}[Finite order-five test]\label{prop:toeplitz5-finite}
Assume $T_5\in\TP$. Let $f$ be as in \eqref{eq:toeplitz5-f} and put
\[
g(t)=e^{4\alpha t}f(t),\qquad
\mathcal C=\{s>0:g'(s)=0\}.
\]
Then $\mathcal C$ contains at most eleven points, and
\[
T_5\in\ETP\quad\Longleftrightarrow\quad
\begin{gathered}
G_5\succeq0,\quad T_4\in\ETP,\quad C_*\in\ETP,\\
g(s)>0\ \text{for every }s\in\mathcal C.
\end{gathered}
\]
The condition $C_*\in\ETP$ is equivalently the explicit inequality
\[
(\alpha+\gamma-2\beta)(\eta-2\beta)
\geq(\alpha+\eta-\beta-\gamma)(\gamma-\alpha-\beta).
\]
\end{proposition}

\begin{proof}
Total positivity implies
$\alpha>0$, $\beta>2\alpha$, $\gamma>2\beta-\alpha$,
and $\eta>2\gamma-\beta$.
Expanding \eqref{eq:toeplitz5-f} gives at most fifteen exponential terms.
The unique smallest decay exponent is $4\alpha$, with coefficient one;
all other decay exponents exceed $4\alpha$ by the inequalities just stated.
Consequently $g(0)=0$ and $g(t)\to1$ as $t\to\infty$.
The derivative $g'$ is a nonzero exponential polynomial with at most fourteen
terms after equal exponents are combined. A nonzero exponential polynomial
with $m$ distinct real exponents has at most $m-1$ real zeros, counted with
multiplicity, as follows by induction using Rolle's theorem after removing
one exponential factor; see also \cite{PolyaSzego1976}.
The zero of $g'$ at the origin has multiplicity at least two, because
$f(t)=O(t^3)$. Thus $g'$ has at most eleven positive zeros.

If $g$ is negative anywhere in $(0,\infty)$, its value at zero and its positive
limit imply that it attains a negative minimum at an interior point.
If $g$ is nonnegative but vanishes at a positive point, that point is a
critical point with value zero. Hence $g>0$ on $(0,\infty)$ if and only if
it is positive on $\mathcal C$.
The four contiguous blocks of order three of $B_4$, in top-left,
top-right, bottom-left, and bottom-right order, are
$B_3,C_*,T_3,B_3$. If $T_4\in\ETP$ and $C_*\in\ETP$, they are all ETP:
$B_3$ and $T_3$ are submatrices of $T_4$ and hence are ETP; the 
assumption $C_*\in\ETP$ supplies the remaining block. Every smaller contiguous square submatrix of $B_4$
is contained in one of these order-three blocks. Thus all its minors
required by the contiguous-minor criterion, except its full determinant,
are positive for every $t>0$. Under these hypotheses,
$B_4\in\ETP$ if and only if $f(t)>0$ for every $t>0$.
Since $e^{4\alpha t}>0$, this is equivalent to $g>0$ on $(0,\infty)$,
and hence to the stated critical-point test. Together with $G_5\succeq0$,
Corollary~\ref{cor:toeplitz-recursion} proves sufficiency.
Conversely, if $T_5\in\ETP$, its blocks $T_4$, $C_*$, and $B_4$ are ETP,
the same corollary gives $G_5\succeq0$, and $g(s)>0$ at every
$s\in\mathcal C$. Finally, the explicit condition for $C_*$ follows
from Theorem~\ref{thm:classification} and its logarithmic core above.
\end{proof}

\begin{remark}\label{rem:toeplitz5-twozeros}
For Example~\ref{ex:toeplitz5-counterexample}, the decay exponents in
increasing order and their coefficients are
\[
\begin{array}{c|rrrrrrrr}
\lambda&4&5.83&7.66&9.32&9.66&11.32&13.15&14.32\\
c_\lambda&1&-3&1&2&2&-2&-2&-1
\end{array}
\]
\[
\begin{array}{c|rrrrrrr}
\lambda&15.32&16.32&16.98&18.64&20.15&20.47&21.98\\
c_\lambda&-1&2&2&1&-2&-1&1
\end{array}
\]
There are six sign changes, so Laguerre's rule bounds the total multiplicity
of real zeros of $f$ by six. The origin has multiplicity exactly three.
The function is positive near zero and for large $t$, but negative at $1/25$.
It therefore has at least two positive zeros. Their total multiplicity must
be even, since the signs at the two ends agree. The remaining bound of three
shows that there are exactly two positive zeros, both simple.
Their approximate values are
\[
t_-=0.001138565875\ldots,\qquad t_+=0.061455497185\ldots.
\]
Thus $f(t)<0$ exactly on $(t_-,t_+)$. All principal minors of every positive
power are positive, but a nonprincipal minor fails in this intermediate range.
\end{remark}

Proposition~\ref{prop:toeplitz5-finite} is a finite test involving zeros of an
exponential polynomial. It does not express the answer by polynomial
inequalities in the logarithmic parameters. Example~\ref{ex:toeplitz5-counterexample}
shows why the semidefinite and leading-coefficient conditions alone do not
provide such an answer.

%\begin{remark}[Open problem]\label{open:toeplitz5}
%Find an explicit necessary and sufficient condition on
%$(\alpha,\beta,\gamma,\eta)$ for a symmetric totally positive Toeplitz matrix
%of order five to be eternally totally positive. In particular, determine
%whether, within the TP domain, this condition can be expressed by finitely
%many polynomial inequalities in these logarithmic parameters.
%\end{remark}

\section{A finite Hankel form of the classical infinite-divisibility criterion}
\label{sec:hankel-ETP}

For Hankel matrices, the translated block that controls nonprincipal minors
is again symmetric. This allows a nonrecursive criterion in every finite
order. We present it as a finite strict-total-positivity formulation of
classical infinite-divisibility results, not as a new characterization of
infinitely divisible moment sequences.

Berg \cite[Theorem~2.4]{Berg2007} characterizes the positive normalized sequences
whose every positive power is a Stieltjes moment sequence, through an integral
representation of their logarithms. Benhida, Curto, and Exner
\cite[Proposition~1.4 and Theorem~2.13]{BenhidaCurtoExner2021} explicitly use
\[
\delta_k=\log\frac{\gamma_k\gamma_{k+2}}{\gamma_{k+1}^2}
\]
and the conversion of conditional positivity into positivity of second
finite differences. Their weighted-shift setting involves infinite sequences.
Below we require only a fixed finite Hankel matrix and its translated block,
and use the strictness hypothesis to obtain positivity of every minor.
The terminology in \cite{BenhidaCurtoExner2021} calls semidefinite matrices
``positive definite''; here we distinguish $\succeq0$ from $\succ0$.

\begin{lemma}[Classical strict Hankel criterion]\label{lem:hankel-two-blocks}
Let $H_n=(h_{i+j})_{i,j=0}^{n-1}$ be a real Hankel matrix, $n\geq2$, and set
$H_n^+=(h_{i+j+1})_{i,j=0}^{n-2}$.
Then $H_n\in\TP$ if and only if $H_n\succ0$ and $H_n^+\succ0$.
\end{lemma}

\begin{proof}
The result follows immediately from Sylvester's criterion and the contiguous-minor criterion \cite{FallatJohnson2011, Pinkus2010}.
\end{proof}

\paragraph{Example illustrating Lemma~\ref{lem:hankel-two-blocks}.}
Let $h_k=1^k+2^k+3^k+4^k$ for $0\leq k\leq6$. Then
\[
(h_0,\ldots,h_6)=(4,10,30,100,354,1300,4890),
\]
and the two matrices in the lemma are
\[
H_4=\begin{pmatrix}
4&10&30&100\\
10&30&100&354\\
30&100&354&1300\\
100&354&1300&4890
\end{pmatrix},\qquad
H_4^+=\begin{pmatrix}
10&30&100\\30&100&354\\100&354&1300
\end{pmatrix}.
\]
Their leading principal minors, in increasing order of size, are
\[
\begin{array}{c|rrrr}
H_4&4&20&80&144\\
H_4^+&10&100&840&
\end{array}.
\]
Sylvester's criterion gives $H_4\succ0$ and $H_4^+\succ0$;
Lemma~\ref{lem:hankel-two-blocks} therefore gives $H_4\in\TP$.
To illustrate the contiguous-block step of the proof, the determinants of
its contiguous blocks of orders two and three are, respectively,
\[
\begin{pmatrix}
20&100&620\\100&620&4684\\620&4684&41060
\end{pmatrix},\qquad
\begin{pmatrix}80&840\\840&9936\end{pmatrix}.
\]
Here entry $(p,q)$ records the determinant of the block beginning at
row $p$ and column $q$, with indices starting at zero. Blocks with even
$p+q$ are principal blocks of $H_4$; those with odd $p+q$ are principal
blocks of $H_4^+$. The order-one minors are the positive entries of $H_4$,
and $\det H_4=144>0$, so every contiguous minor is positive.
Conversely, the resulting total positivity includes the positive leading
principal minors of both matrices used above. This example illustrates
strict TP; eternal TP requires the additional logarithmic test below.

\begin{theorem}[Finite strict Hankel criterion for ETP]
\label{thm:hankel-finite-ETP}
Let $H_n=(h_{i+j})_{i,j=0}^{n-1}$, $n\geq2$, be totally positive. Define
\[
\ell_k=\log h_k,\qquad
d_k=\ell_k-2\ell_{k+1}+\ell_{k+2}
=\log\frac{h_kh_{k+2}}{h_{k+1}^2}
\quad(0\leq k\leq2n-4),
\]
and
\begin{equation}\label{eq:hankel-difference-blocks}
D_0=(d_{i+j})_{i,j=0}^{n-2},\qquad
D_1=(d_{i+j+1})_{i,j=0}^{n-3}.
\end{equation}
For $n=2$, $D_1$ is empty and its semidefinite condition is omitted.
The following conditions are equivalent:
\begin{enumerate}[label=\textup{(\roman*)}]
\item $H_n\in\ETP$;
\item $\log^{\circ}H_n$ and $\log^{\circ}H_n^+$ are conditionally positive
semidefinite;
\item $D_0\succeq0$ and $D_1\succeq0$.
\end{enumerate}
\end{theorem}

\begin{proof}
Apply Lemma~\ref{lem:hankel-two-blocks} to each positive Hadamard power.
Condition (i) is equivalent to positive definiteness of
$H_n^{\circ t}$ and $(H_n^+)^{\circ t}$ for every $t>0$.
The principal minors of order two of both original matrices are strictly
positive, because they are minors of $H_n$. Hence
Lemma~\ref{lem:toeplitz-strict-ID} proves equivalence with (ii).
If $n=2$, the second matrix is a positive scalar, so its conditional-positivity
condition is vacuous and its powers are positive.

To prove (ii)$\iff$(iii), for $m\geq2$ let $E_m\in\R^{m\times(m-1)}$
have columns $e_0-e_1,\ldots,e_{m-2}-e_{m-1}$, where the $e_i$ are
the standard basis vectors of $\R^m$. Each column has zero sum.
The columns are linearly independent and hence form a basis of the
$(m-1)$-dimensional zero-sum subspace. Thus every $v\in\R^m$ with
$\mathbf1^{\mathsf T}v=0$ can be written uniquely as $v=E_m u$, and
$v^{\mathsf T}Mv=u^{\mathsf T}(E_m^{\mathsf T}ME_m)u$.

A real symmetric $m\times m$ matrix $M$ is therefore conditionally
positive semidefinite if and only if $E_m^{\mathsf T}ME_m\succeq0$.
For $0\leq i,j\leq n-2$, the Hankel entries give
\begin{align*}
\bigl(E_n^{\mathsf T}(\log^{\circ}H_n)E_n\bigr)_{ij}
&=(e_i-e_{i+1})^{\mathsf T}(\log^{\circ}H_n)(e_j-e_{j+1})\\
&=\ell_{i+j}-\ell_{i+j+1}-\ell_{i+j+1}+\ell_{i+j+2}\\
&=d_{i+j}.
\end{align*}
Similarly, for $n\geq3$ and $0\leq i,j\leq n-3$,
\[\bigl(E_{n-1}^{\mathsf T}(\log^{\circ}H_n^+)E_{n-1}\bigr)_{ij}
=\ell_{i+j+1}-2\ell_{i+j+2}+\ell_{i+j+3}
=d_{i+j+1}.\]
Consequently,
$E_n^{\mathsf T}(\log^{\circ}H_n)E_n=D_0$ and
$E_{n-1}^{\mathsf T}(\log^{\circ}H_n^+)E_{n-1}=D_1, (n\geq3)$.
For $n=2$, only the first identity is needed, as noted above.
This proves (ii)$\iff$(iii). The matrices in (iii) may be singular: strict positivity
of the powers is supplied by the strict order-two minors in the hypothesis,
as in Lemma~\ref{lem:toeplitz-strict-ID}.
\end{proof}

\begin{remark}[Finite data versus an infinite moment sequence]
Theorem~\ref{thm:hankel-finite-ETP} tests only $h_0,\ldots,h_{2n-2}$ and
requires no extension to an infinite infinitely divisible Stieltjes moment
sequence. We do not assert such an extension. Its finite-difference step is
precisely the classical conditional-positivity mechanism discussed in
\cite{BenhidaCurtoExner2021}. The role of the present formulation is to make
that mechanism an explicit test for strict ETP of a finite TP Hankel matrix.
\end{remark}

\begin{corollary}[Finite polynomial test in logarithmic coordinates]
Under the hypotheses of Theorem~\ref{thm:hankel-finite-ETP}, ETP is equivalent
to the nonnegativity of every principal minor of $D_0$ and $D_1$.
These are finitely many polynomial inequalities in the variables
$\ell_0,\ldots,\ell_{2n-2}$, or equivalently in the $d_k$.
\end{corollary}

\begin{proof}
Use the principal-minor characterization of positive semidefiniteness:
all principal minors must be tested, not only the leading principal minors.
The entries of $D_0,D_1$ are linear in the logarithmic variables.
\end{proof}

\begin{example}[A singular difference matrix with strict ETP]
Let $h_k=\exp(u+vk+ck^2)$, where $u,v\in\R$ and $c>0$.
Then $d_k=2c$, so both matrices in \eqref{eq:hankel-difference-blocks}
are positive semidefinite of rank one when nonempty. Moreover,
$h_{i+j}=e^u e^{vi+ci^2}e^{2cij}e^{vj+cj^2}$.

The generalized Vandermonde property makes $(e^{2cij})$ strictly totally
positive, and every positive Hadamard power has the same form. Thus these
Hankel matrices are ETP in every order despite the singular difference
matrices. This is the familiar quadratic-logarithm family associated with
lognormal moments; it is an illustration of the classical theory.
\end{example}

The contrast with Toeplitz is structural. For Hankel, the translated block
is symmetric and the two semidefinite tests settle the finite ETP question.
For symmetric Toeplitz, the translated block is generally nonsymmetric;
Example~\ref{ex:toeplitz5-counterexample} shows that even positivity of all
principal minors of every power does not control that block.

\section{Conclusion}

The Hadamard threshold contains information that a universal critical exponent
cannot detect.  For $3\times3$ matrices, we obtain an exact logarithmic
criterion, explicit bounds, all possible threshold values, and continuity on
the normalized totally positive cone.  For symmetric $4\times4$ matrices, the
logarithmic core separates an easily controlled case from the obstructed one,
and a two-parameter Hankel family realizes every possible threshold.  The
exact ranges are therefore $[0,1)$ in order three and
$[0,1)\cup(1,2)$ in the symmetric order-four case.

The symmetric Toeplitz family also admits an explicit order-four ETP
classification. In order five, infinite divisibility and positivity near the
zero exponent do not control all nonprincipal minors. A finite critical-point
test gives an exact analytic criterion; a classification by explicit
parameter inequalities remains open. For finite TP Hankel matrices, the
classical infinite-divisibility machinery gives two semidefinite tests on
second logarithmic differences, with no remaining nonprincipal obstructions.

\end{document}